\documentclass[11pt, a4paper]{amsart}

\usepackage{amsfonts,amsmath,amssymb, amscd}
\usepackage{txfonts, mathabx}

\newtheorem{theorem}{Theorem}[section]
\newtheorem{lemma}[theorem]{Lemma}

\newtheorem{prop}[theorem]{Proposition}
\newtheorem{coro}[theorem]{Corollary}

\theoremstyle{definition}
\newtheorem{rem}[theorem]{Remark}

\newtheorem{exa}[theorem]{Example}

\newcommand\BB{\mathcal B}
\newcommand\card{\mathrm{card}}
\newcommand\NN{\mathbb{N}}
\newcommand\ZZ{\mathbb Z}
\newcommand\QQ{\mathbb{Q}}
\newcommand\CC{\mathbb{C}}
\newcommand\FF{\mathbb{F}}

\newcommand\PP{\mathbb{P}}
\renewcommand\AA{\mathbb{A}}

\DeclareMathOperator{\Hilb}{Hilb}
\DeclareMathOperator{\Spec}{Spec}

\numberwithin{equation}{section}

\title[Counting ideals]
{Counting the ideals of given codimension of\\ 
the algebra of Laurent polynomials in two variables}

\author{Christian Kassel}
\address{Christian Kassel: 
Universit\'e de Strasbourg, CNRS, IRMA UMR 7501, F--67000 Strasbourg, France}
\email{kassel@math.unistra.fr}
\urladdr{www-irma.u-strasbg.fr/\raise-2pt\hbox{\~{}}kassel/}

\author{Christophe Reutenauer}
\address{Christophe Reutenauer:
Math\'ematiques, Universit\'e du Qu\'ebec \`a Montr\'eal,
Montr\'eal, CP 8888, succ.\ Centre Ville, Canada H3C 3P8}
\email{reutenauer.christophe@uqam.ca}
\urladdr{www.lacim.uqam.ca/\raise-2pt\hbox{\~{}}christo/}

\keywords{Laurent polynomial ring, ideal, partition, divisor, enumeration, Hilbert scheme}

\subjclass[2010]{(Primary)
05A17, 
13F20, 
14C05, 
14N10, 
16S34 
(Secondary)
05A30, 
11P84, 
11T55, 
13P10, 
14G15. 
}

\begin{document}

\begin{abstract}
We establish an explicit formula for the number~$C_n(q)$ of ideals of codimension~$n$ of 
the algebra $\FF_q[x,y,x^{-1}, y^{-1}]$ of Laurent polynomials in two variables over a finite field~$\FF_q$ of cardinality~$q$. 
This number is a palindromic polynomial of degree~$2n$ in~$q$. 
Moreover, $C_n(q) = (q-1)^2 P_n(q)$, where $P_n(q)$ is another palindromic polynomial;
the latter is a $q$-analogue of the sum of divisors of~$n$, 
which happens to be the number of subgroups of~$\ZZ^2$ of index~$n$.
\end{abstract}

\maketitle

\section{Introduction}

Let $\FF_q$ be a finite field of cardinality~$q$ and $\FF_q[x,y,x^{-1}, y^{-1}]$ be the algebra of
Laurent polynomials in two variables with coefficients in~$\FF_q$.

Our main aim is to give a formula for the number $C_n(q)$ of ideals of codimension~$n$
of~$\FF_q[x,y,x^{-1}, y^{-1}]$. Our main result is the following.

\begin{theorem}\label{th-Cnq}
For each integer $n\geq1$ we have
\begin{equation*}
C_n(q) = \sum_{\lambda \,\vdash n} \,
(q-1)^{2 v(\lambda)} \, q^{n - \ell(\lambda)} \,\prod_{i= 1, \ldots, t \atop d_i \geq 1} \, \frac{q^{2d_i} - 1}{q^2 - 1} \, ,
\end{equation*}
where the sum runs over all partitions~$\lambda$ of~$n$. 
The expression~$C_n(q)$ is a monic polynomial of degree~$2n$ in the variable~$q$ with integer coefficients.
Moreover, the polynomial~$C_n(q)$ is divisible by~$(q-1)^2$.
\end{theorem}

The notation $\ell(\lambda)$, $\nu(\lambda)$, $d_i$ appearing in the formula will be explained 
in Section\,\ref{ssec-cells}.
The proof of the theorem will be given in Section\,\ref{proof-th-inv-cell};
it relies on a parametrization by Conca and Valla\,\cite{CV} of the affine cells in the Ellingsrud--Str\o mme
decomposition of the Hilbert scheme of $n$~points on the affine plane.

Note that since $C_n(q)$ is divisible by~$(q-1)^2$, we may define for each $n\geq 1$
a unique polynomial~$P_n(q)$ by
\begin{equation}\label{def-Pnq}
C_n(q) = (q-1)^2 P_n(q),
\end{equation}
which clearly implies $C_n(1) = 0$ for all $n\geq 1$.
Table\,\ref{tableC} (resp.\ Table\,\ref{tableP}) at the end of the paper 
displays the polynomials~$C_n(q)$ (resp.\ the polynomials~$P_n(q)$) for $n \leq 12$.

Theorem\,\ref{th-Cnq} has two interesting consequences.
The first one concerns the polynomials~$P_n(q)$. Let us state it.

\begin{coro}\label{cor-Pnq}
For each $n\geq 1$ the polynomial $P_n(q)$ is a monic polynomial of degree~$2n-2$ with integer coefficients
and we have 
\begin{equation*}
P_n(1) = \sigma(n) = \sum_{d | n\, ;\, d\geq1} \, d .
\end{equation*}
\end{coro}

As is well known, the sum~$\sigma(n)$ of positive divisors of~$n$
is equal to the number of subgroups of index~$n$ of the free abelian group~$\ZZ^2$ of rank~two. 
Thus Theorem\,\ref{th-Cnq} and Corollary\,\ref{cor-Pnq} imply that the number of ideals of codimension~$n$ of 
the Laurent polynomial algebra~$\FF_q[x,y,x^{-1},y^{-1}]$, i.e. of the algebra of the group~$\ZZ^2$,
is, up to the factor $(q-1)^2$, a $q$-analogue\footnote{By a $q$-analogue of an integer~$r$ we mean
a polynomial $P(q)$ in the variable~$q$ such that $P(1) = r$.} 
of the number of subgroups of index~$n$ of~$\ZZ^2$.

A similar phenomenon had been observed by Bacher and the second-named author in\,\cite{BR2}:
up to a power of~$q-1$, the number of right ideals of codimension~$n$
of the algebra~$\FF_q[F_2]$ of the rank~two free group~$F_2$ is a
$q$-analogue of the number of subgroups of index~$n$ of~$F_2$.
Actually it was this observation that prompted us to compute the number of ideals of codimension~$n$
of the algebra~$\FF_q[\ZZ^2]$ of the free abelian group~$\ZZ^2$, i.e. of~$\FF_q[x,y,x^{-1}, y^{-1}]$.

In a similar context, the following holds.

(a) By~\cite{ES1} (see also Section~\ref{ssec-cells} below)
the number of ideals of codimension~$n$ of the polynomial algebra~$\FF_q[x,y]$, 
which is the algebra of the free abelian monoid~$\NN^2$, is
a $q$-analogue of the number~$p(n)$ of partitions of~$n$; 
as is well known,
the latter is equal to the number of ideals of the monoid~$\NN^2$
whose complement is of cardinality~$n$.

(b) In a non-commutative setting, by~\cite{Re, BR1}, the number of right ideals of codimension~$n$ 
of the free algebra~$\FF_q\langle x,y\rangle$
is a $q$-analogue of the number of right ideals of the free monoid $\langle x,y\rangle^*$
whose complement is of cardinality~$n$.

(c) It may be shown that the number of right ideals of codimension~$2$ of the algebra~$\FF_q[F_3]$ of
the rank~three free group~$F_3$ is equal to 
\[
q^2(q-1)^5 \left( (q+1)^3-1 \right).
\] 
The last factor is obviously a $q$-analogue of $2^3-1 = 7$, which is the number of subgroups of index~$2$ of~$F_3$. 

We conjecture the number of right ideals of codimension~$2$ of the algebra~$\FF_q[F_r]$ of 
the free group~$F_r$ with $r$~generators to be of the form $q^i(q-1)^j \left((q+1)^r-1 \right)$ for some non-negative integers $i,j$; 
the last factor is then a $q$-analogue of the number $2^r-1$ of subgroups of index~$2$ of~$F_r$.
More generally, we expect the number of right ideals of codimension~$n$ of~$\FF_q[F_r]$, up to a power of $q-1$, to be 
a $q$-analogue of the number of subgroups of index~$n$ of~$F_r$ 
(see also the conclusion of\,\cite{BR2}).

\begin{rem} 
The commutative algebra $L_r = \FF_q[x_1,x_1^{-1},\ldots, x_r,x_r^{-1}]$ of Laurent polynomials in $r$~variables ($r\geq 3$)
provides a distinct contrast with the cases discussed above. 
We can show that the number of right ideals of codimension~$2$ of~$L_r$, 
which is the algebra of the free abelian group~$\ZZ^r$, is equal to $(q-1)^r R_r(q)$, where 
\[
R_r(q)=\frac{1}{2} \left( (q+1)^r+(q-1)^r \right)+\frac{q^r-1}{q-1}-1.
\]
The latter is a $q$-analogue of $R_r(1)=2^{r-1}+r-1$. Now the number of subgroups of index~$2$ of~$\ZZ^r$
is equal to~$2^r-1$, which is different from~$R_r(1)$ when $r\geq 3$.
\end{rem}

The second consequence of Theorem\,\ref{th-Cnq} expresses the generating function of the polynomials~$C_n(q)$
as a nice infinite product.

\begin{coro}\label{cor-inf_prod}
(a) We have
\begin{equation*}
1 + \sum_{n\geq 1} \, \frac{C_n(q)}{q^n}  \, t^n
= \prod_{i\geq 1}\, \frac{(1-t^i)^2}{1-(q+q^{-1})t^i + t^{2i}} .
\end{equation*}

(b) The polynomials $C_n(q)$ and $P_n(q)$ are palindromic.
\end{coro}

The previous infinite product shows up in\,\cite[p.\,10]{Fi} (see for instance Equations\,(9.2) and\,(10.1))
and probably in other papers on basic hypergeometric series;
in an algebraic geometry context it appears in\,\cite[Th.\,4.1.3]{HLR2}, 
where it is equal to the generating function
of the $E$-polynomials of the punctual Hilbert schemes of the complex two-dimensional torus
(see details in Section\,\ref{ssec-altern} below).

Using Corollary\,\ref{cor-inf_prod}, we gave explicit expressions for the coefficients
of the polynomials~$C_n(q)$ and~$P_n(q)$ in the companion paper\,\cite{KRinf}
(see Theorems\,1.1 and\,1.2 in \emph{loc. cit.}). 
We obtained a rather striking positivity result, namely the coefficients of~$P_n(q)$ are all \emph{non-negative} integers. 
For the sake of completeness we recall our formulas for the coefficients
of the polynomials~$C_n(q)$ and~$P_n(q)$ in Appendix\,\ref{app}.
\\

The paper is organized as follows. Section\,\ref{sec-prelim} is devoted to some preliminaries:
we first recall the one-to-one correspondence between the ideals of the localization~$S^{-1}A$ of an algebra~$A$
and certain ideals of~$A$; we also count tuples of polynomials subject to certain constraints over a finite field.

In Section\,\ref{sec-cells} we recall Conca and Valla's parametrization of the affine cells in 
a decomposition of the Hilbert scheme of $n$~points in the plane;
these cells are indexed by the partitions of~$n$.
We show how to deduce a parametrization of the cells in the induced decomposition of the Hilbert scheme 
of $n$~points in a Zariski open subset of the plane.

In Section\,\ref{sec-semi-inv} we apply the techniques of the preceding section to compute the number of
ideals of codimension~$n$ of~$\FF_q[x,y,y^{-1}]$. 
In passing we give a criterion (Proposition\,\ref{prop-criterion-y}) which will also be used in the proof of Theorem~\ref{th-Cnq}.

In Section\,\ref{sec-inv} we define what we call an invertible Gr\"obner cell, 
which is a Zariski open subset of the corresponding affine cell, 
and compute its cardinality over a finite field. We derive a proof of Theorem~\ref{th-Cnq}. 

The proofs of Corollary\,\ref{cor-inf_prod} of and of Corollary\,\ref{cor-Pnq} are given in Section\,\ref{sec-proof-main}.

In Appendix\,\ref{app} we briefly recall the results on the coefficients of~$C_n(q)$ and~$P_n(q)$ 
we obtained in\,\cite{KRinf}.

\section{Preliminaries}\label{sec-prelim}

We fix a ground field~$k$. By algebra we mean an associative unital $k$-algebra. 
In this paper all algebras are assumed to be \emph{commutative}.

\subsection{Ideals in localizations}\label{ssec-localization}

Let $A$ be a (commutative) algebra, $S$ a multiplicative submonoid of~$A$ not containing~$0$,
and $S^{-1}A$ the corresponding localization of~$A$. 
We assume that the canonical algebra map $i: A \to S^{-1}A$ is injective (this is the case, for instance, when $A$ is a domain).

Recall the well-known correspondence between the ideals of~$S^{-1}A$ and those of~$A$ 
(see \cite[Chap.~2, \S~2, n$^{\text{o}}$~4--5]{Bo},~\cite[Prop.\,2.2]{Eis}).

\begin{itemize}
\item[(a)]
For any ideal~$J$ of~$S^{-1}A$, the set $i^{-1}(J) = J \cap A$ is an ideal of~$A$ and we have
$J = i^{-1}(J)S^{-1}A$.
The map $J \mapsto i^{-1}(J)$ is an injection from the set of ideals of~$S^{-1}A$ to the set of ideals of~$A$.

\item[(b)]
An ideal $I$ of~$A$ is of the form $i^{-1}(J)$ for some ideal~$J$ of~$S^{-1}A$ if and only if for all $s\in S$
the endomorphism of~$A/I$ induced by the multiplication by~$s$ is injective.
\end{itemize}

Given an integer $n\geq 1$, a $n$-\emph{codimensional} ideal of $A$ is an ideal such that $\dim_k A/I = n$. 
For such an ideal, the previous condition\,(b) is then equivalent to: 
for all $s\in S$, the endomorphism of~$A/I$ induced by the multiplication by~$s$ is a linear isomorphism.

We leave the proof of the following lemma to the reader.

\begin{lemma}
If $J$ is a finite-codimensional ideal of~$S^{-1} A$, then the canonical algebra map $i: A \to S^{-1}A$ induces an
algebra isomorphism
\begin{equation*}
A/i^{-1}(J) \cong (S^{-1}A)/J.
\end{equation*}
\end{lemma}

It follows that there is a bijection between the set of $n$-codimensional ideals of~$S^{-1}A$ 
and the set of $n$-codimensional ideals~$I$ of~$A$ such that for all $s\in S$, 
the endomorphism of~$A/I$ induced by the multiplication by~$s$ is a linear isomorphism.
The latter assertion is equivalent to $s$ being invertible modulo~$I$, that is the image of~$s$ in~$A/I$
being invertible.

The following criterion will be used in Sections\,\ref{sec-semi-inv} and\,\ref{sec-inv}.

\begin{lemma}\label{lem-quot-by-s}
Let $A$ be a commutative algebra. For any $s \in A$,
let $p: A \to A/(s)$ be the natural projection onto the quotient algebra of~$A$ by the ideal generated by~$s$.
If $I$ is an ideal of~$A$, then $s$ is invertible modulo~$I$ if and only if $p(I) = A/(s)$.
\end{lemma}

\begin{proof}
If $s$ is invertible modulo~$I$, then there exists $t\in A$ such that $st - 1 \in I$. 
Hence, $p(1)$ belongs to~$p(I)$, which implies $p(I) = A/(s)$.
Conversely, if $p(I) = A/(s)$, then $p(1) = p(u)$ for some $u\in I$. Hence $1 - u \in (s)$,
which means that there is $t\in A$ such that $1 - u = st$. Thus, $st \equiv 1 \pmod{I}$.
\end{proof}

\subsection{Counting polynomials over a finite field}\label{ssec-counting}

In this subsection we assume that $k = \FF_q$ is a finite field of cardinality~$q$.
We shall need the following in Section\,\ref{sec-inv}.

\begin{prop}\label{prop-counting}
Let $d,h$ be integers~$\geq 1$ and $Q_1, \ldots, Q_h \in \FF_q[y]$ be coprime polynomials.
The number of~$(h+1)$-tuples $(P,P_1, \ldots, P_h)$ satisfying the three conditions
\begin{itemize}
\item[(i)]
$P$ is a degree~$d$ monic polynomial with $P(0) \neq 0$,

\item[(ii)]
$P_1, \ldots, P_h$ are polynomials of degree $<d$, and

\item[(iii)]
$P$ and $P_1 Q_1 + \cdots + P_hQ_h$ are coprime,
\end{itemize}
is equal to
\begin{equation*}
(q-1)^2 \, q^{(h-1)d} \, \frac{q^{2d} - 1}{q^2 - 1} \, .
\end{equation*}
\end{prop}

Before giving the proof, we state and prove two auxiliary lemmas.

\begin{lemma}\label{lem-counting}
Let $R$ be a finite commutative ring and $a_1, \ldots, a_h \in R$ such that $a_1 R + \cdots + a_h R = R$.
For any $b\in R$, the number of $h$-tuples $(x_1, \ldots, x_h) \in R^h$ such that 
$a_1 x_1 + \cdots + a_h x_h = b$ is equal to~$(\card\ R)^{h-1}$.
\end{lemma}

\begin{proof}
The map $(x_1, \ldots, x_h) \mapsto a_1 x_1 + \cdots + a_h x_h$ is a homomorphism $R^h \to R$ of
additive groups. 
Since it is surjective, the number of $h$-tuples satisfying the above condition is equal to the cardinality of its kernel,
which  is equal to $\card\ R^h/\card\ R = (\card\ R)^{h-1}$.
\end{proof}

\begin{lemma}\label{lem-counting3}
Let $d \geq 1$ be an integer.
The number of couples $(P,Q) \in \FF_q[y]^2$ such that $P$ is a degree~$d$ monic polynomial with $P(0) \neq 0$, 
$Q$ is of degree~$<d$, and $P$ and $Q$ are coprime is equal to 
\[
c_d = (q-1)^2 \, \frac{q^{2d}-1}{q^2-1}  \, .
\]
\end{lemma}

\begin{proof}
This amounts to counting the number of couples $(P,z)$, where $P \in \FF_q[y]$ is a degree~$d$ monic polynomial
not divisible by~$y$ and $z$ is an invertible element of the quotient ring~$\FF_q[y]/(P)$.

Expanding $P$ into a product of irreducible polynomials and using the Chinese remainder lemma, we have
\begin{equation*}
1 + \sum_{d\geq 1} \, c_d t^d 
= \prod_{P \; \text{irreducible} \atop P \neq y} \, \left(1+ \sum_{k\geq 1} \, \card (\FF_q[y]/(P))^{\times} \, t^{k \deg(P)}\right),
\end{equation*}
where the product is taken over all irreducible polynomials of~$\FF_q[y]$ different from~$y$ and 
where $\deg(P)$ denotes the degree of~$P$. 
First observe that for any irreducible polynomial $P\in \FF_q[y]$
the group $(\FF_q[y]/(P))^{\times}$ of invertible elements of~$\FF_q[y]/(P)$ is of cardinality
$q^{k \deg(P)} - q^{(k-1) \deg(P)}$: indeed, there are $q^{k \deg(P)}$ polynomials of degree $<k \deg(P)$
and $q^{(k-1) \deg(P)}$ of them are divisible by~$P$, hence not invertible in~$\FF_q[y]/(P)$.
Consequently, 
\begin{eqnarray*}
1 + \sum_{d\geq 1} \, c_d t^d 
& = & \prod_{P \; \text{irreducible} \atop P \neq y} \, \left( 1 + \left(1 - q^{-\deg(P)} \right) \sum_{k\geq 1} \, (qt)^{k \deg(P)}\right) \\
& = & \prod_{P \; \text{irreducible} \atop P \neq y} \, \left( 1 + \left(1 - q^{-\deg(P)} \right) \frac{(qt)^{\deg(P)}}{1 - (qt)^{\deg(P)}} \right) \\
& = & \prod_{P \; \text{irreducible} \atop P \neq y} \, \frac{1 -t^{\deg(P)}}{1 - (qt)^{\deg(P)}} \, .
\end{eqnarray*}
On one hand the infinite product $\prod_{P \; \text{irreducible} \atop P \neq y} \, (1 -t^{\deg(P)})^{-1}$ is equal to the 
zeta function $Z_{\AA^1 \setminus\{0\}}(t)$ of the affine line minus a point.
On the other, 
\begin{equation*}
Z_{\AA^1 \setminus\{0\}}(t) = \frac{Z_{\AA^1}(t)}{Z_{\{0\}}(t)} = \frac{1-t}{1-qt} \, .
\end{equation*}
Therefore, 
\begin{equation*}
1 + \sum_{d\geq 1} \, c_d t^d = \frac{1-qt}{1-q^2t}\left/ \frac{1-t}{1-qt} \right.
= \frac{(1-qt)^2}{(1-t)(1-q^2t)} \, .
\end{equation*}
Subtracting~$1$ from both sides, we obtain
\begin{equation*}
\sum_{d\geq 1} \, c_d t^d = (q-1)^2 \frac{t}{(1-t)(1-q^2t)} \, ,
\end{equation*}
from which it is easy to derive the desired formula for~$c_d$.
\end{proof}

\begin{proof}[Proof of Proposition\,\ref{prop-counting}]
We have to count the number of those $(h+2)$-tuples $(P,Q,P_1, \ldots, P_h)$ such that 
$P$ is a degree~$d$ monic polynomial with $P(0) \neq 0$, 
$Q$ is a polynomial of degree~$<d$ and coprime to~$P$, 
each polynomial $P_i$ is of degree~$<d$, and
$\sum_{i=1}^h \, P_iQ_i \equiv Q$ modulo~$P$. 

By Lemma\,\ref{lem-counting3}, the number of couples $(P,Q)$ satisfying these conditions
is equal to $(q-1)^2 \, (q^{2d}-1)/(q^2-1)$. Since $\card\ \FF_q[y]/(P) = q^d$, by Lemma\,\ref{lem-counting}
we have $q^{d(h-1)}$~choices for the $h$-tuples $(P_1, \ldots, P_h)$.
The number we wish to count is the product of the two previous ones.
\end{proof}

\section{The Hilbert scheme of points in a Zariski open subset of the plane}\label{sec-cells}

Let $k$ be a field.
As is well known, the ideals of codimension~$n$ of an affine $k$-algebra~$A$ are in bijection 
with the $k$-points of the Hilbert scheme parametrizing finite subschemes of co\-length~$n$ of the spectrum of~$A$.
For instance the ideals of codimension~$n$ of the polynomial algebra~$k[x,y]$ 
are in bijection with the $k$-points of the Hilbert scheme~$\Hilb^n(\AA^2_k)$ of $n$~points on the affine plane.
Similarly, the ideals of codimension~$n$ of the Laurent polynomial algebra~$k[x,y,x^{-1},y^{-1}]$
are in bijection with the $k$-points of the Hilbert scheme~$\Hilb^n((\AA^1_k \setminus\{0\}) \times (\AA^1_k \setminus\{0\}))$
of $n$~points on the two-dimensional torus, which is a Zariski open subset of the plane.

In this paragraph we prove that the Hilbert scheme of $n$~points in a Zariski open subset of the plane
is an open subscheme of the Hilbert scheme of $n$~points in the plane, and show how to determine it explicitly.

\subsection{Parametrizing the finite-codimensional ideals of~$k[x,y]$}\label{ssec-cells}

Computing the homology of Hilbert scheme~$\Hilb^n(\AA^2_k)$, El\-lings\-rud and Str\o mme\,\cite{ES1} showed that it
has a cellular decomposition indexed by the partitions~$\lambda$ of~$n$,
each cell~$C_{\lambda}$ being an affine space of dimension $n+\ell(\lambda)$, where $\ell(\lambda)$
is the length of~$\lambda$. 

It follows that, in the special case when $k = \FF_q$ is a finite field of cardinality~$q$, 
the number~$A_n(q)$ of ideals of~$\FF_q[x,y]$ of codimension~$n$ is finite and given by the polynomial
\begin{equation}\label{Anq}
A_n(q) = \sum_{\lambda \,\vdash n}\, q^{n+\ell(\lambda)},
\end{equation}
where the sum runs over all partitions~$\lambda$ of~$n$ 
(we indicate this by the notation $\lambda \vdash n$ or by $|\lambda| = n$).
The polynomial $A_n(q)$ clearly has non-negative integer coefficients, its degree is~$2n$, 
and $A_n(1) = p(n)$ is equal to the number of partitions of~$n$ 
(for more on the polynomials~$A_n(q)$, see Remark\,\ref{rem-A}).

For our purposes we need an explicit description of the affine cells~$C_{\lambda}$.
We use a parametrization due to Conca and Valla\,\cite{CV}.
Let us now recall it. 

Given a positive integer~$n$, there is a well-known bijection between the partitions of~$n$
and the monomials ideals of codimension~$n$ of~$k[x,y]$.
The correspondence is as follows: to a partition~$\lambda$ of~$n$ we associate 
the sequence 
\begin{equation*}\label{def-m}
0 = m_0 < m_1 \leq \cdots \leq m_t 
\end{equation*}
of integers counting from right to left 
the boxes in each column of the Ferrers diagram of~$\lambda$; 
we have $m_1 + \cdots + m_t = n$.
Then the associated monomial ideal~$I_{\lambda}^0$ is given by
\begin{equation}\label{eq-monomial}
I_{\lambda}^0 = (x^t, x^{t-1}y^{m_1}, \ldots, xy^{m_{t-1}}, y^{m_t}).
\end{equation}
(Note that the generating set in the right-hand side of~\eqref{eq-monomial} is in general not minimal.)
The set $\BB_{\lambda} = \{x^i y^j \; |\, 0 \leq i < t, \; 0 \leq j <m_i \}$ induces a linear basis of the 
$n$-dimensional quotient algebra~$k[x,y]/I_{\lambda}^0$. 

Consider the lexicographic ordering on the monomials $x^i y^j$ given by
\[
1 < y < y^2 < \cdots < x < xy < xy^2 < \cdots < x^2 < x^2y < x^2y^2 < \cdots
\]
Then the cell~$C_{\lambda}$, called \emph{Gr\"obner cell} in\,\cite{CV}, is
by definition the set of ideals~$I$ of~$k[x,y]$ such that the dominating terms (for this ordering) of the elements of~$I$
generate the monomial ideal~$I_{\lambda}^0$.
It was proved in\,\cite{ES1} that $C_{\lambda}$ is an affine space.

Here is how Conca and Valla explicitly parametrize~$C_{\lambda}$.
Given a partition~$\lambda$ of~$n$ and the associated sequence $0 = m_0 < m_1 \leq \cdots \leq m_t$,
they first define the sequence of integers $d_1, \ldots, d_t$ by
\begin{equation}\label{def-d}
d_i = m_i - m_{i-1} \geq 0.
\end{equation}
We have $d_1 = m_1 >0$.

Later we shall also need the integer
\begin{equation}\label{def-v}
v(\lambda) = \card \left\{i = 1, \ldots, t \; |\; d_i \geq 1\right\};
\end{equation}
this integer is equal to the number of distinct values of the sequence $m_1 \leq \cdots \leq m_t$.
Note that $v(\lambda) \geq 1$; moreover, $v(\lambda) = 1$ if and only if the partition is ``rectangular'',
i.e. $m_1 = \cdots =  m_t \, (>0)$.

Let $T_{\lambda}$ be the set of $(t+1) \times t$-matrices $(p_{i,j})$ with entries in the one-variable polynomial  algebra~$k[y]$ 
satisfying the following conditions: $p_{i,j} = 0$ if $i <j$, 
the degree of~$p_{i,j}$ is less than~$d_j$ if $i \geq j$ and $d_j \geq 1$,
and $p_{i,j} = 0$ for all $i$ if $d_j = 0$.
The set $T_{\lambda}$ is an affine space whose dimension is $n+ \ell(\lambda)$.

Now consider the $(t+1) \times t$-matrix
{\scriptsize
\begin{equation}\label{matrixM}
M_{\lambda} = 
\begin{pmatrix}
y^{d_1} + p_1 & 0 & 0 &\cdots  & 0 & 0 &0 & \cdots & 0\\
p_{2,1} - x & y^{d_2} + p_2 & 0 & \cdots  & 0 & 0 & 0 & \cdots & 0\\
p_{3,1}  & p_{3,2} - x & y^{d_3} + p_3  & \cdots  & 0 & 0 & 0 & \cdots & 0\\
\vdots  & \vdots & \vdots  & \ddots  & \vdots & \vdots & \vdots &  &\vdots \\
p_{i-1,1}  & p_{i-1,2} & p_{i-1,3} & \cdots  &  y^{d_{i-1}} + p_{i-1}  & 0 & 0 & \cdots & 0 \\
p_{i,1}  & p_{i,2} & p_{i,3} & \cdots  &  p_{i,i-1} - x & y^{d_i} + p_i & 0 & \cdots & 0 \\
p_{i+1,1}  & p_{i+1,2} & p_{i+1,3} & \cdots  &  p_{i+1,i-1}  & p_{i+1,i} - x & y^{d_{i+1}} + p_{i+1} & \cdots & 0 \\
\vdots  & \vdots & \vdots  & \ddots  & \vdots & \vdots & \vdots & \ddots & \vdots \\
p_{t,1}  & p_{t,2} & p_{t,3} & \cdots  &  p_{t,i-1}  & p_{t,i} & p_{t,i+1}  & \cdots & y^{d_t} + p_{t}\\
p_{t+1,1}  & p_{t+1,2} & p_{t+1,3} & \cdots  &  p_{t+1,i-1}  & p_{t+1,i} & p_{t+1,i+1}  & \cdots & p_{t+1,t} - x
\end{pmatrix} 
,
\end{equation}
}%
where for simplicity we set $p_i = p_{i,i}$.

By\,\cite[Th.\,3.3]{CV} the map sending the polynomial matrix~$(p_{i,j}) \in T_{\lambda}$ 
to the ideal~$I_{\lambda}$ of~$k[x,y]$ generated by all $t$-minors (the maximal minors)
of the matrix~$M_{\lambda}$ is a bijection of $T_{\lambda}$ onto~$C_{\lambda}$.
These minors are polynomial expressions with integer coefficients in the coefficients of the $p_{i,j}$'s.

\subsection{Localizing}\label{ssec-loc}

Let $S$ be a multiplicative submonoid of~$k[x,y]$ not containing~$0$.
We assume that $S$ has a finite generating set~$\Sigma$. 
In the sequel we shall concentrate on two cases:  
$\Sigma = \{y\}$ (in Section\,\ref{sec-semi-inv}) and $\Sigma = \{x,y\}$ (in Section\,\ref{sec-inv}).

It follows from Section\,\ref{sec-prelim} that the set of $n$-codimensional ideals of
the localization~$S^{-1}k[x,y]$ can be identified
with the subset of~$\Hilb^n(\AA^2_k)$ consisting of the $n$-codimensional ideals~$I$ of~$k[x,y]$ 
such that for all $s\in S$, the endomorphism~$\mu_s$ of~$k[x,y]/I$ induced by the multiplication by~$s$ is a linear isomorphism.
The latter is equivalent to $\det\mu_s \neq 0$ for all $s\in \Sigma$.

By the considerations of Section\,\ref{ssec-cells}, the set of $n$-codimensional ideals of the algebra~$S^{-1}k[x,y]$
is the disjoint union 
\begin{equation*}
\coprod_{\lambda \,\vdash n} \, C_{\lambda}^{\Sigma} \, ,
\end{equation*}
where $C_{\lambda}^{\Sigma}$ is the Zariski open subset of the affine Gr\"obner cell~$C_{\lambda}$
consisting of the points satisfying $\det\mu_s \neq 0$ for all $s\in \Sigma$.

Consequently, the Hilbert scheme $\Hilb^n(\Spec(S^{-1}k[x,y]))$ parametrizing subschemes of colength~$n$
in~$\Spec(S^{-1}k[x,y])$ is an open subscheme of~$\Hilb^n(\AA^2_k)$, hence an open subscheme of~$\Hilb^n(\PP^2_k)$.
Since by\,\cite{Fo, Gr} the latter is smooth and projective, 
$\Hilb^n(\Spec(S^{-1}k[x,y]))$ is a smooth quasi-projective variety.

The endomorphism $\mu_x$ (resp.\ $\mu_y$) of~$k[x,y]/I$ induced by the multiplication by~$x$ (resp.\ by~$y$)
can be expressed as a matrix in the basis~$\BB_{\lambda}$. 
Observe that the entries of such a matrix are polynomial expressions with integer coefficients in the coefficients of the $p_{i,j}$'s.
Therefore, if any $s\in \Sigma$ is a linear combination with integer coefficients of monomials in the variables $x,y$, then 
the Hilbert scheme $\Hilb^n(\Spec(S^{-1}k[x,y]))$ is defined over~$\ZZ$ as a variety.

In particular, the Hilbert schemes $\Hilb^n(\AA^1_k \times (\AA^1_k \setminus \{0\}))$
and $\Hilb^n((\AA^1_k \setminus \{0\})^2)$ are smooth quasi-projective varieties
defined over~$\ZZ$.

\begin{exa}
Let $\lambda$ be the unique self-conjugate partition of~$3$.
In this case, $t=2$, $m_1 = 1$, $m_2 = 2$,
hence $d_1 = d_2 = 1$.
The corresponding matrix~$M_{\lambda}$, as in\,\eqref{matrixM}, is 
\begin{equation*}
M_{\lambda} =
\begin{pmatrix}
y + a & 0\\
b - x & y + d\\
c & e-x
\end{pmatrix}
,
\end{equation*}
where $a,b,c,d,e$ are scalars.
The associated Gr\"obner cell~$C_{\lambda}$ is a $5$-dimensional affine space parametrized by these five scalars.
The ideal~$I_{\lambda}$ is generated by the maximal minors of the matrix, namely by
$(b-x)(e-x)-c(y+d)$, $(e-x)(y+a)$, and $(y+a)(y+d)$.
It follows that modulo~$I_{\lambda}$ we have the relations
\begin{equation*}
x^2 \equiv (b+e)x + cy + (cd - be), \quad
xy \equiv -ax + ey + ae, \quad
y^2 \equiv -(a+d)y -ad.
\end{equation*}

In the basis $\BB_{\lambda} = \{x,y,1\}$ the multiplication endomorphisms $\mu_x$ and $\mu_y$ can be expressed as the matrices
\begin{equation*}
\mu_x =
\begin{pmatrix}
b+e & -a & 1 \\
c & e & 0 \\
cd-be & ae & 0
\end{pmatrix}
\quad
\text{and}
\quad
\mu_y = 
\begin{pmatrix}
-a & 0 & 0 \\
e & -(a+d) & 1 \\
ae & -ad & 0
\end{pmatrix}
.
\end{equation*}
We have $\det\mu_x = e(ac - cd + be)$ and $\det\mu_y = -ad^2$.

It follows from the above computations that, if for instance $\Sigma = \{x,y\}$,
then $C_{\lambda}^{\Sigma}$ is the complement in the affine space~$\AA_k^5$ of the union
of the three hyperplanes $a = 0$, $d= 0$, $e = 0$ and of the quadric hypersurface $ac - cd + be = 0$.
\end{exa}

\section{The punctual Hilbert scheme of the complement of a line in an affine plane}\label{sec-semi-inv}

In this section we apply the considerations of the previous section to the case $\Sigma = \{y\}$.
Here $S$ is the multiplicative submonoid of~$k[x,y]$ generated by~$y$
and $S^{-1}k[x,y] = k[x,y,y^{-1}] = k[x][y,y^{-1}]$.

By Section\,\ref{ssec-loc}, the Hilbert scheme $\Hilb^n(\AA^1_k \times (\AA^1_k \setminus \{0\}))$,
that is the set of $n$-codimensional ideals of $k[x,y,y^{-1}]$,
is the disjoint union over the partitions~$\lambda$ of~$n$ of the sets $C_{\lambda}^{y}$, 
where $C_{\lambda}^{y}$ consists of the ideals $I \in C_{\lambda}$ such that $y$ is invertible in~$k[x,y]/I$. 
We call $C_{\lambda}^{y}$ the \emph{semi-invertible Gr\"obner cell} associated to the partition~$\lambda$.

\subsection{A criterion for the invertibility of~$y$}\label{ssec-criterion-invy}

Let $p_y: k[x,y] \to k[x]$ be the algebra map sending $x$ to itself and $y$ to~$0$.
Then by Lemma\,\ref{lem-quot-by-s},
the set $C_{\lambda}^{y}$ consists of the ideals $I \in C_{\lambda}$ such that $p_y(I) = k[x]$.

Recall from Section\,\ref{ssec-cells} that $I_{\lambda}$ is generated 
by the maximal minors of the matrix~$M_{\lambda}$ of\,\eqref{matrixM},
namely by the polynomials $f_0(x,y), \ldots$, $f_t(x,y)$, 
where we define $f_i(x,y)$ to be the determinant of the $t\times t$-matrix obtained from~$M_{\lambda}$
by deleting its $(i+1)$-st row.
Then the ideal $p_y(I_{\lambda})$ can be identified with the ideal of~$k[x]$
generated by the polynomials $f_0(x,0), \ldots, f_t(x,0) \in k[x]$ obtained by setting $y=0$.
We need to determine under what conditions this ideal is equal to the whole algebra~$k[x]$.

Recall the entries of the matrix~$M_{\lambda}$ and particularly the polynomials $p_{i,j}$ and
$p_i = p_{i,i} \in k[y]$.
Let $a_{i,j} = p_{i,j}(0)$ be the constant term of~$p_{i,j}$. As above, we set $a_i = a_{i,i} = p_i(0)$.
Note that $a_j = 1$ and $a_{i,j} = 0$ for all $i \neq j$ whenever $d_j = 0$.

Then $f_0(x,0), \ldots, f_t(x,0)$ are the maximal minors of the matrix
{\footnotesize
\begin{equation*}
M_{\lambda}^y = 
\begin{pmatrix}
a_1 & 0 & 0 &\cdots  & 0 & 0 &0 & \cdots & 0\\
a_{2,1} - x & a_2 & 0 & \cdots  & 0 & 0 & 0 & \cdots & 0\\
a_{3,1}  & a_{3,2} - x & a_3  & \cdots  & 0 & 0 & 0 & \cdots & 0\\
\vdots  & \vdots & \vdots  & \ddots  & \vdots & \vdots & \vdots &  &\vdots \\
a_{i-1,1}  & a_{i-1,2} & a_{i-1,3} & \cdots  &  a_{i-1}  & 0 & 0 & \cdots & 0 \\
a_{i,1}  & a_{i,2} & a_{i,3} & \cdots  &  a_{i,i-1} - x & a_i & 0 & \cdots & 0 \\
a_{i+1,1}  & a_{i+1,2} & a_{i+1,3} & \cdots  &  a_{i+1,i-1}  & a_{i+1,i} - x & a_{i+1} & \cdots & 0 \\
\vdots  & \vdots & \vdots  & \ddots  & \vdots & \vdots & \vdots & \ddots & \vdots \\
a_{t,1}  & a_{t,2} & a_{t,3} & \cdots  &  a_{t,i-1}  & a_{t,i} & a_{t,i+1}  & \cdots & a_{t}\\
a_{t+1,1}  & a_{t+1,2} & a_{t+1,3} & \cdots  &  a_{t+1,i-1}  & a_{t+1,i} & a_{t+1,i+1}  & \cdots & a_{t+1,t} - x
\end{pmatrix} .
\end{equation*}
}%
To be precise, $f_i(x,0)$ is the determinant of the square matrix obtained from~$M_{\lambda}^y$
by deleting its $(i+1)$-st row.

The criterion we need is the following.

\begin{prop}\label{prop-criterion-y}
We have $p_y(I_{\lambda}) = k[x]$ if and only if $a_i \neq 0$ for all $i = 1, \ldots, t$ such that $d_i \geq 1$. 
\end{prop}

\begin{proof}
Since $a_i = 1$ when $d_i = 0$, it is equivalent to prove that $p_y(I_{\lambda}) = k[x]$ if and only if 
$a_{1} a_{2} \cdots a_{t} \neq 0$.

Set $I_x = p_y(I_{\lambda}) \subset k[x]$.
The condition $a_{1} a_{2} \cdots a_{t} \neq 0$ is sufficient.
Indeed, the last polynomial, $f_t(x,0)$, 
is the determinant of a lower triangular matrix whose diagonal entries are the
scalars~$a_i$; hence, $f_t(x,0) = a_{1} a_{2} \cdots a_{t}$. 
Thus, if $f_t(x,0)$ is non-zero, then $I_x = k[x]$.

To check the necessity of the condition, we will prove that for each $i= 1, \ldots, t$,
the vanishing of the scalar~$a_i$ implies that the ideal~$I_x$ is contained in a proper ideal generated
by a minor of~$M_{\lambda}^y$.

If $a_1 = 0$, then $f_1(x,0) =  \cdots = f_t(x,0) = 0$ since these are determinants of matrices whose first row is zero.
It follows that $I_x$ is the principal ideal generated by the characteristic polynomial~$f_0(x,0)$, which is of degree~$t\geq 1$.
Hence, $I_x$ is a proper ideal of~$k[x]$.

Let now $i \geq 2$. If for $k \geq i$, we delete the $(k+1)$-st row of~$M_{\lambda}^y$,
we obtain a lower block-triangular matrix of the form
\[
\begin{pmatrix}
M_1 & 0 \\
\ast & M_2^{(k)}
\end{pmatrix} 
,
\]
where $M_1$ is the square submatrix of~$M_{\lambda}^y$ corresponding to the rows $1, \ldots, i$
and to the columns $1, \ldots, i$;
this is a lower triangular matrix whose diagonal entries are $a_1, \ldots, a_i$.
Consequently, if $a_i = 0$, then $f_k(x,0) = 0$ for all $k \geq i$.

Under the same condition $a_i = 0$, if we delete the $(k+1)$-st row of~$M_{\lambda}^y$ for $k < i$,
then we obtain a lower block-triangular matrix of the form
\[
\begin{pmatrix}
M_1^{(k)} & 0 \\
\ast & M_2
\end{pmatrix} 
,
\]
where $M_2$ is the square submatrix of~$M_{\lambda}^y$ corresponding to the rows $i+1, \ldots t+1$ and to the columns $i, \ldots t$:
\[
M_2 = 
\begin{pmatrix}
a_{i+1,i} - x & a_{i+1} & \cdots & 0 & 0\\
a_{i+2,i}  & a_{i+2,i+1} -x & \cdots & 0 & 0\\
\vdots & \vdots & \ddots & \vdots & \vdots\\
a_{t,i}  & \cdots & \cdots & a_{t,t-1} - x & a_t\\
a_{t+1,i} & a_{t+1,i+1}  & \cdots & a_{t+1,t-1} & a_{t+1,t} - x
\end{pmatrix}
.
\]
Consequently, the polynomials $f_k(x,0)$ for $k<i$ are all divisible by the determinant of~$M_2$.
Thus, $I_x$ is contained in the ideal generated by~$\det (M_2)$,
which is a characteristic polynomial of degree~$t-i+1$.
Since $t-i+1 \geq 1$ for all $i= 1, \ldots, t$, we have $I_x \neq k[x]$.
\end{proof}

As an immediate consequence of Section\,\ref{ssec-loc} and of Proposition\,\ref{prop-criterion-y}
we obtain the following.

\begin{coro}\label{coro-y-inverse}
The set of $n$-codimensional ideals of~$k[x,y,y^{-1}]$ is the disjoint union 
\begin{equation*}
\coprod_{\lambda \,\vdash n} \, C_{\lambda}^y \, ,
\end{equation*}
where $C_{\lambda}^y$ is the complement in the affine Gr\"obner cell~$C_{\lambda}$
of the union of the hyperplanes $a_i = 0$ where $i$ runs over all integers $i= 1, \ldots, t$ such that $d_i \geq 1$.
\end{coro}

\subsection{On the number of finite-codimensional ideals of~$\FF_q[x,y,y^{-1}]$}\label{ssec-semi-invert}

Recall the positive integer~$v(\lambda)$ defined by\,\eqref{def-v}.

\begin{prop}\label{prop-semi-inv}
Let $k= \FF_q$. For each partition~$\lambda$ of~$n$, the set $C_{\lambda}^y$ is finite
and its cardinality is given by
\begin{equation*}
\card\ C_{\lambda}^y = (q-1)^{v(\lambda)} \, q^{n + \ell(\lambda) - v(\lambda)}.
\end{equation*}
\end{prop}

\begin{proof}
By Corollary\,\ref{coro-y-inverse} the set $C_{\lambda}^y$ 
is parametrized by $n + \ell(\lambda)$ parameters subject to the sole condition that $v(\lambda)$ of them are not  zero.
\end{proof}

\begin{coro}\label{coro-ideaux-y-inverse}
For each integer $n\geq 1$, 
the number~$B_n(q)$ of $n$-codimensional ideals of~$\FF_q[x,y,y^{-1}]$ is equal to $(q-1) \, q^n  \, B^{\circ}_n(q)$, where
\begin{equation*}
B^{\circ}_n(q) = \sum_{\lambda \,\vdash n}\, (q-1)^{v(\lambda)-1} \, q^{\ell(\lambda) - v(\lambda)}.
\end{equation*}
\end{coro}

Note that $B^{\circ}_n(q)$ is a polynomial in~$q$ since $v(\lambda) \geq 1$ and $\ell(\lambda) \geq v(\lambda)$ for all partitions.
It is of degree~$n-1$ and has integer coefficients.
The coefficients of~$B^{\circ}_n(q)$ may be negative, as one can see in Table\,\ref{tableB} at the end of the paper.

\begin{rem}
Let $v_n$ be the valuation of the polynomial~$B^{\circ}_n(q)$, i.e.\
the maximal integer $r$ such that $q^r$ divides~$B^{\circ}_n(q)$.
We conjecture that $v_n = 0$, $1$, or~$2$,
and that the infinite word $v_1v_2v_3 \ldots$ is equal to $0 \prod_{n=1}^{\infty} \, 0 1^{2n} 0 2^n $.
\end{rem}

Let us now give a product formula for the generating function of the sequence of polynomials~$B_n(q)$
and an arithmetical interpretation for two values of~$B_n^{\circ}(q)$.

\begin{theorem}\label{th-B}
(a) Let $B_n(q)$ be the number of ideals of~$\FF_q[x,y,y^{-1}]$ of codimension~$n$.
We have
\begin{equation*}
1 + \sum_{n\geq 1}\, \frac{B_n(q)}{q^n} \, t^n = \prod_{i\geq 1}\, \frac{1-t^i}{1 - qt^i} \, .
\end{equation*}

(b) Let $B_n^{\circ}(q)$ be the polynomial $B_n^{\circ}(q) = (q-1)^{-1}q^{-n} B_n(q)$. 
It has integer coefficients and satisfies
\begin{equation*}
B_n^{\circ}(1) = \sigma_0(n),
\end{equation*}
where $\sigma_0(n)$ is the number of divisors of~$n$, and 
\begin{equation*}
B_n^{\circ}(-1) = 
\left\{
\begin{array}{cl}
(-1)^{k-1} & \text{if}\; n = k^2 \, \text{ for some integer}\; k, \\
0 \quad &  \text{otherwise.}
\end{array}
\right.
\end{equation*}
\end{theorem}

\begin{proof}
(a) Since an analogous proof will be used in Remark~\ref{rem-A} and in Section~\ref{ssec-pf-cor-inf_prod}, 
we give here a detailed proof.
Let $X$ be a set and~$M$ be the free abelian monoid on~$X$ ($X$ is called a basis of~$M$).
We say that a function $\varphi: M \to R$ from~$M$ to a ring~$R$ is \emph{multiplicative} 
if $\varphi(uv) = \varphi(u) \varphi(v)$ for all couples $(u,v) \in M^2$ of words having no common basis element. 
Under this condition, it is easy to check the following identity:
\begin{equation}\label{eq-mult}
\sum_{w\in M}\, \varphi(w) = \prod_{x\in X} \, \left( 1 + \sum_{e\geq 1} \, \varphi(x^e) \right).
\end{equation}

Now, identifying each partition with its planar diagram, 
we consider a partition~$\lambda$ as a union of rectangular
partitions~$i^{e_i}$, with $e_i$~parts of length~$i$, for $e_i \geq 1$ and distinct $i\geq 1$,
which we denote by the formal product $\lambda = \prod_{i\geq 1} \, i^{e_i}$.
Thus the set of partitions is equal to the free abelian monoid on~$X = \NN \setminus \{0\}$. 
Before we apply~\eqref{eq-mult}, let us remark that $|\lambda| = \sum_i \, ie_i $ and $\ell(\lambda) = \sum_i\, e_i$.
Moreover, the multisets $\{e_i \, |\, i\geq 1\}$ and $\{d_i \, |\, i\geq 1\}$ are equal
(recall that the integers~$d_i$ are those associated with~$\lambda$ in~\eqref{def-d}); therefore,
$v(\lambda) = \sum_i\, 1 = \card\{i \, |\, e_i \geq 1\}$.

The function $\lambda \mapsto \card\ C_{\lambda}^y \, s^{|\lambda|}$
computed in Proposition\,\ref{prop-semi-inv} is clearly multiplicative. 
Applying~\eqref{eq-mult}, we obtain
\begin{eqnarray*}
1 + \sum_{n\geq 1}\, B_n(q) s^n
& = & 1 + \sum_{|\lambda| \geq 1} \, \card\ C_{\lambda}^y \, s^{|\lambda|} \\
& = & \prod_{i\geq 1}\, \left( 1 + \sum_{e\geq 1}\, \card\ C_{i^e}^y \, s^{ie} \right) \\
& = & \prod_{i\geq 1}\, \left( 1 +  \sum_{e \geq 1}\, (q-1) q^{ie + e-1} s^{ie} \right)\\
& = & \prod_{i\geq 1}\, \left( 1 + (q-1)q^{-1} \sum_{e \geq 1}\, (q^{i+1} s^i)^e \right)\\
& = & \prod_{i\geq 1}\, \left( 1 + (q-1)q^{-1} \frac{q^{i+1} s^i}{1 - q^{i+1} s^i} \right)\\
& = & \prod_{i\geq 1}\, \frac{(1 - q^{i+1} s^i) + (q-1)q^i s^i}{1 - q^{i+1} s^i}\\
& = & \prod_{i\geq 1}\, \frac{1-q^is^i}{1 - q^{i+1}s^i}  \, .
\end{eqnarray*}
Finally replace $s$ by~$q^{-1}t$.

(b) To compute $B_n^{\circ}(1)$ we use the formula of Corollary\,\ref{coro-ideaux-y-inverse}.
Since the values at $q = 1$ of $(q-1)^{v(\lambda)-1}$ is $1$ if $v(\lambda)=1$ and~$0$ otherwise and
since $v(\lambda)=1$ if and only if $m_1 = \cdots = m_t = d$, in which case $dt = n$,
we have 
\begin{equation*}
B_n^{\circ}(1) = \sum_{dt = n}\, 1 = \sum_{d|n} \, 1 = \sigma_0(n).
\end{equation*}

For $B_n^{\circ}(-1)$ we use the infinite product expansion of Item\,(a):
replacing $B_n(q)$ by $(q-1) q^n B_n^{\circ}(q)$, we obtain
\begin{equation*}
1 + \sum_{n\geq 1}\, (q-1) B_n^{\circ}(q) t^n = \prod_{i\geq 1}\, \frac{1-t^i}{1 - qt^i} \, .
\end{equation*}
Setting $q = -1$ yields
\begin{equation*}
1 -2 \sum_{n\geq 1}\, B_n^{\circ}(-1) t^n = \prod_{i\geq 1}\, \frac{1-t^i}{1 + t^i} \, .
\end{equation*}
Now recall the following identity of Gauss (see\,\cite[(7.324)]{Fi} or\,\cite[19.9\,(i)]{HW}):
\begin{equation}\label{eq-Gauss}
\prod_{i\geq 1}\, \frac{1-t^i}{1 + t^i} = \sum_{k\in \ZZ}\, (-1)^k t^{k^2}.
\end{equation}
It follows that 
\begin{equation*}
1 -2 \sum_{n\geq 1}\, B_n^{\circ}(-1) t^n = 1 + 2 \sum_{k \geq 1}\, (-1)^k t^{k^2},
\end{equation*}
which allows us to conclude.
\end{proof}

\begin{rem}\label{rem-A}
The results of Theorem\,\ref{th-B} should be compared to the following ones concerning 
the number~$A_n(q)$ of ideals of~$\FF_q[x,y]$ of codimension~$n$. 
Proceeding as in the proof of Theorem\,\ref{th-B}, we deduce from\,\eqref{Anq} that
\begin{equation*}
1 + \sum_{n\geq 1}\, A_n(q) s^n = \prod_{i\geq 1}\, \frac{1}{1 - q^{i+1}s^i} \, .
\end{equation*}
Setting $q = -1$, we have
\begin{equation}\label{eq-AA}
1 + \sum_{n\geq 1}\, A_n(-1) s^n = \prod_{i\geq 1}\, \frac{1}{1 - (-1)^{i+1}s^i}
= \prod_{m\geq 1}\, \frac{1}{(1 - s^{2m-1})(1 + s^{2m})} .
\end{equation}
Multiplying by $\prod_{m\geq 1}\, (1 + s^{2m})^{-1}$ both sides of the Euler identity
\[
\prod_{m\geq 1}\, \frac{1}{1 - s^{2m-1}} = \prod_{i\geq 1}\, (1+s^i)
\]
(see~\cite[(19.4.7)]{HW}),
we deduce that the right-hand side of~\eqref{eq-AA} is equal to the infinite product
\begin{equation*}
\prod_{m\geq 1}\, (1 + s^{2m-1}).
\end{equation*}%
Thus by\,\cite[Table\,14.1, p.\,310]{Ap} or\,\cite[(19.4.4)]{HW}, 
the value $A_n(-1)$ is equal to the number\footnote{See Sequence~A000700 in\,\cite{OEIS}.} 
of partitions of~$n$ with unequal odd parts.
Note that $A_n(1)$ is equal to the number\footnote{See Sequence~A000041 in\,\cite{OEIS}.} of partitions of~$n$.
See Table\,\ref{tableA} at the end for a list of the polynomials~$A_n(q)$ ($1 \leq n \leq 12$).
\end{rem}

\section{Invertible Gr\"obner cells}\label{sec-inv}

Let $\Hilb^n((\AA^1_k \setminus \{0\})^2)$ be the Hilbert scheme parametrizing finite subschemes of colength~$n$  
of the two-dimensional torus, i.e.\
of the complement of two distinct intersecting lines in the affine plane.
Its $k$-points are in bijection with the set of ideals of $k[x,y,x^{-1},y^{-1}]$  of codimension~$n$.
By Section\,\ref{ssec-loc} this set of ideals is the disjoint union over the partitions~$\lambda$ of~$n$
of the sets $C_{\lambda}^{x,y}$, where $C_{\lambda}^{x,y}$ consists of the ideals $I \in C_{\lambda}$
such that both $x$ and $y$ are invertible in~$k[x,y]/I$. 
We call $C_{\lambda}^{x,y}$ the \emph{invertible Gr\"obner cell} associated to the partition~$\lambda$. 

When the ground field is finite, so is~$C_{\lambda}^{x,y}$. 
The aim of this section is to compute the cardinality of~$C_{\lambda}^{x,y}$ when $k = \FF_q$.

\subsection{The cardinality of an invertible Gr\"obner cell}\label{ssec-invert}

Recall the non-negative integers $d_1, \ldots, d_t$ defined by\,\eqref{def-d}
and the positive integer~$v(\lambda)$ defined by\,\eqref{def-v}.
We now give a formula for~$\card\ C_{\lambda}^{x,y}$.

\begin{theorem}\label{th-inv-cell}
Let $k = \FF_q$, $n$ an integer~$\geq 1$ and $\lambda$ be a partition of~$n$. Then
\begin{equation*}
\card\ C_{\lambda}^{x,y} = 
(q-1)^{2 v(\lambda)} \, q^{n - \ell(\lambda)} \,\prod_{i= 1, \ldots, t \atop d_i \geq 1} \, \frac{q^{2d_i} - 1}{q^2 - 1} \,  . 
\end{equation*}
\end{theorem}

The theorem will be proved in Section\,\ref{proof-th-inv-cell}. It has the following straightforward consequences.

\begin{coro}\label{coro-inv-cell}
Let $k = \FF_q$ and $\lambda$ be a partition of~$n$. 

(a) $\card\ C_{\lambda}^{x,y}$ is a monic polynomial in~$q$ with integer coefficients; it is of degree~$n + \ell(\lambda)$.

(b) The polynomial $\card\ C_{\lambda}^{x,y}$ is divisible by~$(q-1)^2$.
The quotient 
\begin{equation*}
P_{\lambda}(q) = \frac{\card\ C_{\lambda}^{x,y}}{(q-1)^2}
\end{equation*}
is a monic polynomial in~$q$ with integer coefficients and of degree~$n + \ell(\lambda) -2$.

(c) If the partition~$\lambda$ is rectangular, i.e., if $v(\lambda) = 1$, 
in which case $d_2 = \cdots =d_t = 0$ and $d = d_1$ is a divisor of~$n$, then
\begin{equation*}
P_{\lambda}(q) = q^{n - d} \, \frac{q^{2d} - 1}{q^2 - 1}  = q^{n - d} \, \left( 1 + q^2 + \cdots + q^{2d-2} \right). 
\end{equation*}
In this case, $P_{\lambda}(1) = d$.

(d) If $v(\lambda) \geq 2$, then $P_{\lambda}(q)$ is divisible by~$(q-1)^2$, and $P_{\lambda}(1) = 0$.
\end{coro}

\begin{rem}
The polynomials $P_{\lambda}(q)$ may have negative coefficients.
For instance, if $\lambda$ is the partition of~$4$ corresponding to $t=2$, $d_1=1$, $d_2=2$, then
\begin{equation*}
P_{\lambda}(q) = q^5 - 2q^4 + 2q^3 - 2q^2 +q.
\end{equation*}
\end{rem}

The rest of the section is devoted to the proof of Theorem\,\ref{th-inv-cell}.

\subsection{A criterion for the invertibility of~$x$}\label{ssec-criterion-invx}

In Section\,\ref{sec-semi-inv} we introduced the algebra map $p_y: k[x,y] \to k[x]$ sending $x$ to itself and $y$ to~$0$.
Similarly, let $p_x: k[x,y] \to k[x]$ be the algebra map sending $x$ to~$0$ and $y$ to itself.
Then by Lemma\,\ref{lem-quot-by-s},
the set $C_{\lambda}^{x,y}$ consists of the ideals $I \in C_{\lambda}$ such that $p_x(I) = k[y]$ and $p_y(I) = k[x]$.
We already have a criterion for $p_y(I) = k[x]$ (see Proposition\,\ref{prop-criterion-y}). 
We shall now give a necessary and sufficient condition for $p_x(I)$ to be equal to~$k[y]$.

Resuming the notation of Section\,\ref{sec-semi-inv}, we see that $p_x(I)$ can be identified with the ideal of~$k[y]$
generated by the polynomials $f_0(0,y), \ldots, f_t(0,y) \in k[y]$ 
obtained from the polynomials $f_0(x,y), \ldots, f_t(x,y)$ by setting $x=0$.
The polynomials $f_0(0,y), \ldots, f_t(0,y)$ are the maximal minors of the matrix
{\scriptsize
\begin{equation*}
M_{\lambda}^{x} = 
\begin{pmatrix}
y^{d_1} + p_1 & 0 & 0 &\cdots  & 0 & 0 &0 & \cdots & 0\\
p_{2,1} & y^{d_2} + p_2 & 0 & \cdots  & 0 & 0 & 0 & \cdots & 0\\
p_{3,1}  & p_{3,2}  & y^{d_3} + p_3  & \cdots  & 0 & 0 & 0 & \cdots & 0\\
\vdots  & \vdots & \vdots  & \ddots  & \vdots & \vdots & \vdots &  &\vdots \\
p_{i-1,1}  & p_{i-1,2} & p_{i-1,3} & \cdots  &  y^{d_{i-1}} + p_{i-1}  & 0 & 0 & \cdots & 0 \\
p_{i,1}  & p_{i,2} & p_{i,3} & \cdots  &  p_{i,i-1} & y^{d_i} + p_i & 0 & \cdots & 0 \\
p_{i+1,1}  & p_{i+1,2} & p_{i+1,3} & \cdots  &  p_{i+1,i-1}  & p_{i+1,i} & y^{d_{i+1}} + p_{i+1} & \cdots & 0 \\
\vdots  & \vdots & \vdots  & \ddots  & \vdots & \vdots & \vdots & \ddots & \vdots \\
p_{t,1}  & p_{t,2} & p_{t,3} & \cdots  &  p_{t,i-1}  & p_{t,i} & p_{t,i+1}  & \cdots & y^{d_t} + p_{t}\\
p_{t+1,1}  & p_{t+1,2} & p_{t+1,3} & \cdots  &  p_{t+1,i-1}  & p_{t+1,i} & p_{t+1,i+1}  & \cdots & p_{t+1,t} 
\end{pmatrix} 
\end{equation*}
}%
obtained from the matrix~$M_{\lambda}$ of\,\eqref{matrixM} by setting $x = 0$.

Let $\mu_i$ be the determinant of the submatrix~$M_i$ of~$M_{\lambda}^{x}$ corresponding to the rows
$(i+1), \ldots, (t+1)$ and to the columns $i, \ldots, t$. We have
$\mu_t = p_{t+1,t}$ and 
\begin{equation*}
\mu_i =
\begin{vmatrix}
\, p_{i+1,i} & y^{d_{i+1}} + p_{i+1} & \cdots & 0 \\
\vdots & \vdots & \ddots & \vdots \\
p_{t,i} & p_{t,i+1}  & \cdots & y^{d_t} + p_{t} \, \\
\, p_{t+1,i} & p_{t+1,i+1}  & \cdots & p_{t+1,t} 
\end{vmatrix}
\end{equation*}
if $1\leq i < t$.
Expanding~$\mu_i$ along its first column, we obtain
\begin{equation}\label{expand-mu}
\mu_i  = \sum_{j=1}^{t-i+1} \,  p_{i+j,i} \, q_{i+j,i} \, ,
\end{equation}
where 
\begin{equation}\label{cofacteurs-mu}
q_{i+j,i}  = \!\!
\left\{
\begin{array}{cl}
 \mu_{i+1} & \text{if} \, j=1, \\
\noalign{\smallskip}
\hskip -10pt (-1)^{j-1} \, (y^{d_{i+1}} + p_{i+1}) \cdots (y^{d_{i+j-1}} + p_{i+j-1}) \, \mu_{i+j} & \text{if} \, 1 < j < t-i+1, \\
\noalign{\smallskip}
\hskip -5pt  (-1)^{t-i} \, (y^{d_{i+1}} + p_{i+1}) \cdots (y^{d_{t-1}} + p_{t-1}) (y^{d_t} + p_t) & \text{if} \, j = t-i+1.
\end{array}
\right.
\end{equation}

Observe also that
\begin{equation}\label{f(x=0)}
f_i(0,y) = 
\left\{
\begin{array}{ccl}
\mu_1 && \text{if} \; i=0, \\
\noalign{\smallskip}
(y^{d_1} + p_1) \cdots (y^{d_i} + p_i) \, \mu_{i+1} && \text{if} \; 1 \leq i < t,\\
\noalign{\smallskip}
(y^{d_1} + p_1) \cdots (y^{d_t} + p_t) && \text{if} \; i = t.
\end{array}
\right.
\end{equation}

\begin{lemma}\label{lem-mu}
If $1 \leq i \leq  j \leq t$, then $\mu_i$ belongs to the ideal $(\mu_j, y^{d_j} + p_j)$
generated by $\mu_j$ and $(y^{d_j} + p_j)$.
\end{lemma}

\begin{proof}
The case $i=j$ is obvious. 
Otherwise, consider the matrix~$M_i$ whose determinant is~$\mu_i$; the column of~$M_i$
containing the entry $y^{d_j} + p_j$ can be written as the sum of a column containing only the entry $y^{d_j} + p_j$,
the other entries being zero, and of a column whose top entry is zero and the bottom ones form the first column of
the matrix~$M_j$ whose determinant is~$\mu_j$. 
Therefore by the multilinearity property of determinants, 
$\mu_i$ is the sum of a determinant which is a multiple of~$y^{d_j} + p_j$
and of another determinant which is a multiple of~$\mu_j$; 
indeed, this second determinant is block-triangular with one diagonal block equal to~$\mu_j$.
\end{proof}

Here is our criterion for the invertibility of~$x$.

\begin{prop}\label{prop-criterion-x}
We have $p_x(I_{\lambda}) = k[y]$ if and only if $y^{d_i} + p_i$ and~$\mu_i$ are coprime 
for all $i = 1, \ldots, t$.
\end{prop}

\begin{proof}
(a) Let us first check that the above condition is sufficient.
The fact that $y^{d_t} + p_t$ and~$\mu_t$ are coprime implies that by\,\eqref{f(x=0)}
the gcd of $f_t(0,y)$ and of~$f_{t-1}(0,y)$ is $(y^{d_1} + p_1) \cdots (y^{d_{t-1}} + p_{t-1})$. 
Now the gcd of the latter and of~$f_{t-2}(0,y)$ is $(y^{d_1} + p_1) \cdots (y^{d_{t-2}} + p_{t-2})$
in view of the fact that $y^{d_{t-1}} + p_{t-1}$ and~$\mu_{t-1}$ are coprime. Repeating this argument,
we find that the gcd of $f_0(0,y), \ldots, f_t(0,y)$ is~$1$, which implies that $p_x(I_{\lambda}) = k[y]$.

(b) Conversely, suppose that $y^{d_j} + p_j$ and~$\mu_j$ are not coprime for some $j$,
i.e., $(\mu_j, y^{d_j} + p_j) \neq k[y]$. 
By\,\eqref{f(x=0)} and Lemma\,\ref{lem-mu}, $f_0(0,y), \ldots, f_{j-1}(0,y)$ belong to the ideal~$(\mu_j, y^{d_j} + p_j)$.
On the other hand, again by\,\eqref{f(x=0)}, the remaining polynomials $f_j(0,y), \ldots, f_t(0,y)$ are divisible by~$y^{d_j} + p_j$,
hence belong to~$(\mu_j, y^{d_j} + p_j)$.
Therefore, $p_x(I_{\lambda}) \subseteq (\mu_j, y^{d_j} + p_j) \neq k[y]$. 
\end{proof}

For the proof of Theorem\,\ref{th-inv-cell}, we shall also need the following result.

\begin{lemma}\label{lem-coprime}
If $y^{d_j} + p_j$ and~$\mu_j$ are coprime for all $j>i$, then the polynomials $q_{i+1,i}, \ldots, q_{t+1,i}$ 
of~\eqref{cofacteurs-mu} are coprime.
\end{lemma}

\begin{proof}
Proceeding as in Part\,(a) of the proof of Proposition\,\ref{prop-criterion-x} and using\,\eqref{cofacteurs-mu},
one shows by descending induction on~$j$ that
the gcd of $q_{j+1,i}, \ldots, q_{t+1,i}$ is 
\[
(y^{d_{i+1}} + p_{i+1}) \cdots (y^{d_j} + p_j).
\]
In particular, for $j= i+1$, the gcd of $q_{i+2,i}, \ldots, q_{t+1,i}$ is $(y^{d_{i+1}} + p_{i+1})$.
The conclusion follows from this fact 
together with the coprimality of $(y^{d_{i+1}} + p_{i+1})$ and of $q_{i+1,i} = \mu_{i+1}$.
\end{proof}

\subsection{Proof of Theorem\,\ref{th-inv-cell}}\label{proof-th-inv-cell}

By Propositions\,\ref{prop-criterion-y} and\,\ref{prop-criterion-x}, 
it is enough to count the entries of the matrix~$M_{\lambda}$ over~$\FF_q[y]$
such that $p_i(0) \neq 0$ and $y^{d_i} + p_i$ and~$\mu_i$ are coprime for all $i = 1, \ldots t$.
We consider these conditions successively for $i= t, t-1, \ldots, 1$.

Assume first that all integers $d_1, \ldots, d_t$ are non-zero.
For $i=t$, $y^{d_t} + p_t$ is a monic polynomial of degree~$d_t$ with non-zero constant term, 
$\mu_t = p_{t+1,t}$ is of degree~$<d_t$, and both polynomials are coprime. 
It follows from Lemma\,\ref{lem-counting3} (or from Proposition\,\ref{prop-counting} with $d= d_t$ and $h=1$)
that we have $(q-1)^2 (q^{2d_t} - 1)/(q^2 -1)$ possible choices for the last column of~$M_{\lambda}$.

For $i=t-1$, it follows from\,\eqref{expand-mu} that 
$\mu_{t-1} = P_1Q_1 + P_2Q_2$, where $Q_1 = q_{t,t-1}$ and $Q_2 = - q_{t+1,t-1}$,
which are coprime by Lemma\,\ref{lem-coprime},
$P_1 = p_{t,t-1}$ and $P_2 = p_{t+1,t-1}$, which are both polynomials of degree~$<d_{t-1}$. 
The polynomial $P = y^{d_{t-1}} + p_{t-1}$ is monic of degree~$d_{t-1}$ with non-zero constant term,
and $Q = \mu_{t-1} = P_1Q_1 + P_2 Q_2$ is coprime to~$P$ by the coprimality condition. 
It then follows from Proposition\,\ref{prop-counting} applied to the case $d= d_{t-1}$ and $h=2$ that 
there are 
\[
(q-1)^2 q^{d_{t-1}} \, \frac{q^{2d_{t-1}} - 1}{q^2 -1}
\]
possible choices for the $(t-1)$-st column of~$M_{\lambda}$.

In general,  the polynomial $P = y^{d_i} + p_i$ is monic of degree~$d_{t-1}$ with non-zero constant term,
and is assumed to be coprime to $Q = \mu_i  = \sum_{j=1}^{t-i+1} \,  p_{i+j,i} \, q_{i+j,i}$. 
By  Lemma\,\ref{lem-coprime} the polynomials $q_{i+1,i}, \ldots, q_{t+1,i}$ are coprime. 
Applying Proposition\,\ref{prop-counting} to the case $d= d_i$ and $h=t+1-i$, we see that there are
\[
(q-1)^2 q^{(t-i)d_i} \, \frac{q^{2d_i} - 1}{q^2 -1}
\]
possible choices for the $i$-th column of~$M_{\lambda}$.

In the end we obtain a number of possible entries for~$M_{\lambda}$ equal to
\begin{equation*}
\prod_{i=1}^t\, (q-1)^2 q^{(t-i)d_i} \, \frac{q^{2d_i} - 1}{q^2 -1}
= q^{n- \ell(\lambda)} \, \prod_{i=1}^t\, (q-1)^2 \, \frac{q^{2d_i} - 1}{q^2 -1}
\end{equation*}
since $\ell(\lambda) = \sum_{i=1}^t \, d_i$ and $n = |\lambda| = \sum_{i=1}^t \, (t-i+1) \, d_i$.
We have thus proved the theorem when all $d_1, \ldots, d_t$ are non-zero.

Let $E$ be the subset of~$\{1, \ldots, t\}$ consisting of those subscripts~$i$
for which $d_i = 0$. (Note that $1$ does not belong to~$E$ since $d_1 > 0$.)
Assume now that $E$ is non-empty and set $t' = t - \card\ E$. By assumption~$t'< t$.
For any positive integer~$i \leq t'$, let $d'_i$ be equal to the $i$-th non-zero~$d_i$.
The integers $d'_1 = d_1$, $d'_2, \ldots d'_{t'}$ are positive.

Recall that if $i\in E$, then the $i$-th column of the matrix~$M_{\lambda}$ is zero except for the $(i,i)$-entry which is~$1$.
Permuting rows and columns, we may rearrange~$M_{\lambda}$ into a matrix~$M'_{\lambda}$ of the form
\begin{equation*}
M'_{\lambda} =
\begin{pmatrix}
M_{\nu} & 0 \\
N & I_{t-t'}
\end{pmatrix},
\end{equation*}
where $I_{t-t'}$ is an identity matrix of size~$(t-t')$.
The $(t'+1) \times t'$-matrix~$M_{\nu}$ is of the form\,\eqref{matrixM} with $t$ replaced by~$t'$,
the sequence $d_1, \ldots, d_t$ by the shorter sequence $d'_1, \ldots, d'_{t'}$,
and the partition $\lambda$ by the partition $\nu$ associated to the sequence $d'_1, \ldots, d'_{t'}$.

Let $f'_i$ be the determinant of the square matrix obtained from~$M'_{\lambda}$ by deleting its $(i+1)$-st row.
It is clear that up to sign and to reordering the maximal minors $f'_0, \ldots, f'_t$ of~$M'_{\lambda}$ 
are the same as those of~$M_{\lambda}$.
In view of the special form of~$M'_{\lambda}$, observe that
\begin{equation*}
f'_i =
\left\{
\begin{array}{ccl}
f_i^{(\nu)} && \text{if}\; 0 \leq i \leq t',Ê\\
\noalign{\smallskip}
0 && \text{if}\; t' < i \leq t.
\end{array}
\right.
\end{equation*}
where $f_i^{(\nu)}$ is the determinant of the $t'\times t'$-matrix obtained from~$M_{\nu}$ by deleting its $(i+1)$-st row.

The number of possible entries of~$M_{\lambda}$, which is the same as the number of possible entries of~$M'_{\lambda}$,
is then the product of the number of possible entries of~$N$, which is a power of~$q$, and 
of the number of possible entries of~$M_{\nu}$. Since $d'_1, \ldots, d'_{t'}$ are positive, by the first part of the proof,
we know that the number of possible entries of~$M_{\nu}$ is the product of a power of~$q$ by
\begin{equation*}
\prod_{i=1}^{t'}\, (q-1)^2 \, \frac{q^{2d'_i} - 1}{q^2 -1} \, .
\end{equation*}
In other words, the number of possible entries of~$M_{\lambda}$ is
\begin{equation*}
q^c \, \prod_{i= 1, \ldots, t \atop d_i \geq 1} \, (q-1)^2 \, \frac{q^{2d_i} - 1}{q^2 - 1}
\end{equation*}
for some non-negative integer~$c$. 
Now since the invertible Gr\"obner cell~$C_{\lambda}^{x,y}$ is a Zarisky open subset of the affine Gr\"obner cell~$C_{\lambda}$, 
the degree of the previous polynomial in~$q$ must be the same as the degree of the cardinal of~$C_{\lambda}$, 
which is~$q^{n+\ell(\lambda)}$ by Section\,\ref{ssec-cells}. This suffices to establish that $c = n - \ell(\lambda)$
and to complete the proof of the theorem.

\subsection{Proof of Theorem\,\ref{th-Cnq}}\label{ssec-pf-Cnq}

By our remark at the beginning of Section\,\ref{sec-inv}, 
the number~$C_n(q)$ of ideals of $\FF_q[x,y,x^{-1},y^{-1}]$ of codimension~$n$ is given by
\begin{equation}\label{eq-Cnq}
C_n(q) = \sum_{\lambda \,\vdash n} \, \card\ C_{\lambda}^{x,y},
\end{equation}
where $C_{\lambda}^{x,y}$ is the \emph{invertible Gr\"obner cell} associated to the partition~$\lambda$. 
The equality in Theorem\,\ref{th-Cnq} follows then from the formula for $\card\ C_{\lambda}^{x,y}$
given in Theorem\,\ref{th-inv-cell}.

By Corollary\,\ref{coro-inv-cell}\,(a) $\card\ C_{\lambda}^{x,y}$ is 
a monic polynomial which has integer coefficients and whose degree is $n + \ell(\lambda)$.
Therefore, $C_n(q)$ has integer coefficients and its degree is $\max\{ n + \ell(\lambda) \, |\, \lambda \vdash n\}$.
Now $\ell(\lambda)$ is maximal if and only if $\lambda = 1^n$, in which case $\ell(\lambda) = n$.
Therefore $C_n(q)$ is monic and its degree is~$2n$.

Since $\nu(\lambda) \geq 1$, it follows from the formula in Theorem\,\ref{th-inv-cell} that $\card\ C_{\lambda}^{x,y}$
is divisible by $(q-1)^2$ for each invertible Gr\"obner cell. Therefore, the polynomial~$C_n(q)$
is divisible by~$(q-1)^2$.

\section{Proofs of the corollaries}\label{sec-proof-main}

We now start the proofs of Corollary\,\ref{cor-Pnq} and of Corollary\,\ref{cor-inf_prod}.

\subsection{Proof of Corollary\,\ref{cor-Pnq}}\label{ssec-pf-Pnq}
Since $C_n(q)$ and $(q-1)^2$ are both monic with integer coefficients, so is~$P_n(q)$.
The latter is the sum over all partitions of~$n$ of the polynomials~$P_{\lambda}(q)$
(introduced in Corollary\,\ref{coro-inv-cell}\,(b)).
By Corollary\,\ref{coro-inv-cell}\,(c)--(d), we have $P_{\lambda}(1) = 0$ if $v(\lambda) \geq 2$ and,
if $v(\lambda) = 1$, then $\lambda$ is of the form $t^d$, where $dt=n$, in which case
$P_{\lambda}(1) = d$. The desired formula for~$P_n(1)$ follows.

\subsection{Proof of Corollary\,\ref{cor-inf_prod}}\label{ssec-pf-cor-inf_prod}

As in the proof of Theorem\,\ref{th-B} we consider each partition~$\lambda$ as a union of rectangular
partitions~$i^{e_i}$, with $e_i$~parts of length~$i$, for $e_i \geq 1$ and distinct $i\geq 1$.
Recall that $|\lambda| = \sum_i \, ie_i $, $\ell(\lambda) = \sum_i\, e_i$, and $v(\lambda) = \sum_i\, 1$.
To indicate the dependance of~$e_i$ on~$\lambda$, we write $e_i = e_i(\lambda)$.
We then obtain the following statement.

\begin{prop}\label{prop-gf1}
We have the infinite product expansion
\begin{equation*}
1 + \sum_{\lambda} \, \card\ C_{\lambda}^{x,y} \, s_1^{e_1(\lambda)} \, s_2^{e_2(\lambda)} \cdots
= \prod_{i\geq 1}\, \frac{(1-q^i s_i)^2}{(1-q^{i+1} s_i)(1-q^{i-1} s_i)} \,  .
\end{equation*}
\end{prop}

\begin{proof}
Proceeding as in the proof of Theorem\,\ref{th-B} and using Theorem\,\ref{th-inv-cell}, we deduce
that the left-hand side is equal to
\begin{equation*}
1 + \sum_{\lambda} \, \prod_{i\geq 1}\, (q-1)^2 \, \frac{q^{2e_i} -1}{q^2 - 1} q^{ie_i - e_i} \, s_i^{e_i} ,
\end{equation*}
which in turn is equal to
\begin{eqnarray*}
 \prod_{i\geq 1} & & \hskip -22pt
 \left( 1 +  \frac{(q - 1)^2}{q^2 - 1} \, \sum_{e_i \geq 1} \left( (q^{i+1} s_i)^{e_i} -  (q^{i-1} s_i)^{e_i} \right) \right) \\
& = & \prod_{i\geq 1} \, \left( 1 +  \frac{(q - 1)^2}{q^2 - 1} \, 
\left( \frac{q^{i+1} s_i}{1 - q^{i+1} s_i} -  \frac{q^{i-1} s_i}{1 - q^{i-1} s_i} \right) \right) \\
& = & \prod_{i\geq 1} \, \left( 1 +  \frac{(q - 1)^2}{q^2 - 1} \, \frac{(q^2-1) q^{i-1} s_i}{(1 - q^{i+1} s_i)(1 - q^{i-1} s_i)}  \right) \\
& = & \prod_{i\geq 1}\, \left( 1 + \frac{(q-1)^2 q^{i-1} s_i}{(1-q^{i+1} s_i)(1-q^{i-1} s_i)}  \right) \\
& = & \prod_{i\geq 1}\, \frac{(1-q^i s_i)^2}{(1-q^{i+1} s_i)(1-q^{i-1} s_i)} \, .
\end{eqnarray*}
\end{proof}

\begin{proof}[Proof of Corollary\,\ref{cor-inf_prod}]
(a) Replace $s_i$ by $(t/q)^i$ in Proposition\,\ref{prop-gf1}, use\,\eqref{eq-Cnq} and Theorem\,\ref{th-Cnq},
and observe that $(1-q t^i)(1-q^{-1} t^i) = 1-(q+q^{-1})t^i +  t^{2i}$.

(b) The infinite product is clearly invariant under the transformation $q \leftrightarrow q^{-1}$;
thus, $C_n(q^{-1}) = q^{-2n} \, C_n(q)$. Together with $\deg\ C_n(q) = 2n$, this implies that
$C_n(q)$ is palindromic.
The polynomial~$P_n(q)$ is palindromic as a quotient of two palindromic polynomials.
\end{proof}

\subsection{An alternative proof of Corollary\,\ref{cor-inf_prod}\,(a)}\label{ssec-altern}

After we made public a first version of this article, 
we learnt of an alternative geometric approach to the polynomials~$C_n(q)$. 
Indeed, G\"ottsche and Soergel determined the mixed Hodge structure of 
the punctual Hilbert schemes of any smooth complex algebraic surface (see\,\cite[Th.~2]{GS}). 
Applying their result to the Hilbert scheme 
$H^n_{\CC} = \Hilb^n ( \CC^{\times} \times \CC^{\times})$ 
of $n$~points of the complex two-dimensional torus, 
Hausel, Letellier and Rodriguez-Villegas observed in\,\cite[Th.\,4.1.3]{HLR2}
that the compactly supported mixed Hodge polynomial
$H_c(H^n_{\CC}; q,u)$ of~$H^n_{\CC}$ fits into the equality of formal power series
\begin{equation}\label{Hodge}
1 + \sum_{n\geq 1} \, H_c(H^n_{\CC}; q,u) \, \frac{t^n}{q^n} 
= \prod_{i\geq 1} \, \frac{(1+ u^{2i+1}t^i)^2}{(1 - u^{2i+2}qt^i) (1 - u^{2i}q^{-1}t^i)} \, .
\end{equation}
Setting $u = -1$ in\,\eqref{Hodge}, we obtain an infinite product expansion for the generating function 
of the $E$-polynomial $E(H^n_{\CC}; q) = H_c(H^n_{\CC}; q,-1)$ of~$H^n_{\CC}$, namely
\begin{equation}\label{eq-HLR13}
1 + \sum_{n\geq 1} \, E(H^n_{\CC}; q) \, \frac{t^n}{q^n} 
= \prod_{i\geq 1} \, \frac{(1-t^i)^2}{1 - (q+ q^{-1}) t^i +  t^{2i}} \, .
\end{equation}
Now, $H^n_{\CC}$ is strongly polynomial-count in the sense of Nick Katz (see\,\cite[Appendix]{HR}),
probably a well-known fact (which also follows from the computations in the present paper).
Therefore, by~\cite[Th.\,6.1.2]{HR} the $E$-polynomial counts the number of elements of~$H^n$
over the finite field~$\FF_q$, which is also the number~$C_n(q)$ of ideals of codimension~$n$
of~$\FF_q[x,y,x^{-1}, y^{-1}]$. 
Thus \eqref{eq-HLR13} implies the equality of Corollary\,\ref{cor-inf_prod}\,(a).

\begin{rem}
In the same vein as above, there is a geometric explanation of the palindromicity of the polynomials~$C_n(q)$.
In~\cite{CHM} de Cataldo, Hausel, Migliorini observed that any diffeomorphism between $\CC^{\times} \times \CC^{\times}$
and the cotangent bundle $E \times \CC$ of the elliptic curve $E = \CC/\ZZ[i]$ 
induces a linear isomorphism of graded vector spaces
between the cohomology groups of the corresponding Hilbert schemes:
$H^*(H^n_{\CC}, \QQ) \cong H^*(\Hilb^n(E\times \CC), \QQ)$.
This isomorphism does not preserve the mixed Hodge structures, as the one on the right-hand side is pure 
whereas the one on the left-hand side is not. 
Nevertheless, such an isomorphism identifies the weight filtration on $H^*(H^n_{\CC}, \QQ) $
with the perverse Leray filtration on $H^*(\Hilb^n(E\times \CC), \QQ)$ 
associated to the natural projective map from $\Hilb^n(E\times \CC)$ to the $n$-th symmetric product of~$\CC$
induced by the projection of~$E\times \CC$ on the second factor.
The perverse Leray filtration is ``palindromic'' as a consequence of the relative hard Lefschetz theorem for the map above
(see~\cite[Th.\,4.1.1 and Th.\,4.3.2]{CHM}). 

Note that Hausel, Letellier and Rodriguez-Villegas observed a similar palindromicity for the $E$-polynomial of
certain character varieties  and termed it ``curious Poincar\'e duality'' in\,\cite[Cor.\,5.2.4]{HLR1}
(see also \cite[Cor.\,3.5.3]{HR}, \cite[Cor.\,1.4]{HLR0}).
\end{rem}

\begin{rem}
The natural action of the group $\CC^{\times} \times \CC^{\times}$ on itself induces
an action on the Hilbert scheme~$H^n_{\CC}$. Consider the GIT quotient
$\widetilde{H}^n_{\CC} = H^n_{\CC} \, /\!/ (\CC^{\times} \times \CC^{\times})$.
Using~\cite[Th.\,2.2.12]{HR} and \cite[Sect.\,5.3]{HLR1}, 
we see that the $E$-polynomial of~$\widetilde{H}^n_{\CC}$ is given by
\[
E(\widetilde{H}^n_{\CC}; q) = E(H^n_{\CC}; q)/(q-1)^2 = C_n(q)/(q-1)^2 = P_n(q).
\]
Recall from the introduction (see also the appendix below) that the coefficients of~$P_n(q)$ are all non-negative.
Therefore, $\widetilde{H}^n_{\CC}$ provides an example of a polynomial-count variety with odd cohomology 
and a counting polynomial with non-negative coefficients.
This implies non-trivial cancellation for the mixed Hodge numbers of~$\widetilde{H}^n_{\CC}$.
No similar positivity phenomenon was observed for the character varieties investigated 
by Hausel, Letellier and Rodriguez-Villegas.
\end{rem}

\appendix

\section{The coefficients of the polynomials~$C_n(q)$ and $P_n(q)$}\label{app}

We now state the results of the companion paper\,\cite{KRinf} on 
the coefficients of the polynomials~$C_n(q)$ and $P_n(q)$.

Since $C_n(q)$ and $P_n(q)$ are palindromic of respective degrees $2n$ and $2n-2$,
we may expand $C_n(q)$ and $P_n(q)$ as follows:
\begin{equation*}\label{def-cni}
C_n(q) = c_{n,0} \, q^n + \sum_{i=1}^n \, c_{n,i} \, \left( q^{n+i} + q^{n-i} \right),
\end{equation*}
where $c_{n,0}, c_{n,1}, c_{n,2}\ldots$ are integers, and
\begin{equation*}\label{def-ani}
P_n(q) = a_{n,0} \, q^{n-1} + \sum_{i=1}^{n-1} \, a_{n,i} \, \left( q^{n+i-1} + q^{n-i+1} \right),
\end{equation*}
where $a_{n,0}, a_{n,1}, a_{n,2}\ldots$ are integers.

By Theorem\,1.1 of\,\cite{KRinf} the coefficients $c_{n,i}$ of~$C_n(q)$ are given by the following formulas:
(a) For the central coefficients $c_{n,0}$ we have
\begin{equation*}
c_{n,0} =
\left\{
\begin{array}{cl}
2\, (-1)^k  & \text{if}\;\, n = k(k+1)/2 \;\; \text{for some integer}\; k \geq 1, 
\\
\noalign{\smallskip}
0 & \text{otherwise.}
\end{array}
\right.
\end{equation*}
(b) For the non-central coefficients ($i\geq 1$) we have
\begin{equation*}
c_{n,i} =
\left\{
\begin{array}{cl}
(-1)^k & \text{if}\;\, n = k(k+2i +1)/2 \;\; \text{for some integer}\; k  \geq 1, \\
\noalign{\smallskip}
(-1)^{k-1} & \text{if}\;\, n = k(k+2i -1)/2 \;\; \text{for some integer}\; k \geq 1, \\
\noalign{\smallskip}
0 & \text{otherwise.} 
\end{array}
\right.
\end{equation*}
Note that in Item\,(b) the first two conditions are mutually exclusive.

As for the coefficients of~$P_n(q)$, the coefficient~$a_{n,i}$ is by \cite[Th.\,1.2]{KRinf}
equal to the number of divisors~$d$ of~$n$ such that
\begin{equation*}
\frac{i+ \sqrt{2n+i^2}}{2} < d \leq i+ \sqrt{2n+i^2}.
\end{equation*}
It follows that all coefficients~$a_{n,i}$ of~$P_n(q)$ are non-negative integers.

\section*{Acknowledgement}

We are grateful to Olivier Benoist, Fran\c cois Bergeron, Mark Haiman, Emma\-nuel Letellier and Luca Migliorini
for useful discussions, and to Pierre Baumann for suggesting the proof of Lemma\,\ref{lem-counting3}.

The second-named author is grateful to the Universit\'e de Strasbourg for the invited professorship
which allowed him to spend the month of June~2014 at IRMA; 
he was also supported by NSERC (Canada).


\begin{table}[ht]
\caption{\emph{The polynomials $C_n(q)$}}\label{tableC}
\renewcommand\arraystretch{1.25}
\noindent\[
\begin{array}{|c||c|}
\hline
n & C_n(q)  \\
\hline\hline
1 & q^2 - 2q + 1  \\ 
\hline
2 & q^4 - q^3 - q + 1   \\
\hline
3 & q^6 - q^5 - q^4 + 2 q^3 - q^2 - q + 1   \\ 
\hline
4 & q^8 - q^7- q + 1   \\ 
\hline
5 & q^{10} - q^9 - q^7+ q^6 +  q^4 - q^3 - q + 1    \\ 
\hline
6 & q^{12} - q^{11} + q^7 - 2 q^6 + q^5 - q + 1  \\ 
\hline
7 & q^{14} - q^{13} - q^{10} + q^9 + q^5 - q^4 - q + 1    \\ 
\hline
8 &  q^{16} - q^{15} - q + 1  \\ 
\hline
9 &  q^{18} - q^{17} - q^{13} + q^{12} + q^{11} -  q^{10} -  q^8 + q^7+ q^6 - q^5 - q + 1   \\
\hline
10 &  q^{20} - q^{19} - q^{11} + 2q^{10} -  q^9 - q + 1   \\
\hline
11 &  q^{22} - q^{21} - q^{16} + q^{15} + q^7 - q^6 - q + 1   \\
\hline
12 &  q^{24} - q^{23} + q^{15} - q^{14} -  q^{10} + q^9 - q + 1   \\
\hline
\end{array}
\]
\end{table}

\begin{table}[ht]
\caption{\emph{The polynomials $P_n(q)$}}\label{tableP}
\renewcommand\arraystretch{1.3}
\noindent\[
\begin{array}{|c||c|c|}
\hline
n & P_n(q) & P_n(1) \\
\hline\hline
1 & 1 & 1  \\ 
\hline
2 & q^2 + q + 1 & 3  \\
\hline
3 & q^4 + q^3 + q + 1 & 4  \\ 
\hline
4 & q^6 + q^5+ q^4 + q^3 + q^2 + q + 1 & 7 \\ 
\hline
5 & q^8 + q^7+ q^6 + q^2 + q + 1 & 6  \\ 
\hline
 & q^{10} + q^9 + q^8 + q^7+ q^6  &   \\ 
6 & + 2q^5+ q^4 + q^3 + q^2 + q + 1& 12  \\ 
\hline
7 &  q^{12} + q^{11} + q^{10} + q^9 + q^3 + q^2 + q + 1 & 8  \\ 
\hline
 &  q^{14} + q^{13} + q^{12} + q^{11} + q^{10} + q^9 + q^8   &  \\ 
8 &  + q^7 + q^6 + q^5+ q^4 + q^3 + q^2 + q + 1 & 15 \\ 
\hline
 &  q^{16} + q^{15}+ q^{14} + q^{13} + q^{12} + q^9    &  \\ 
9 &  + q^8 + q^7+ q^4 + q^3 + q^2 + q + 1  & 13  \\
\hline
 &  q^{18} + q^{17}+ q^{16} + q^{15}+ q^{14} + q^{13}   &  \\ 
  &  + q^{12} + q^{11}  + q^{10} + q^8 + q^7+ q^6   &   \\
10 &   + q^5 + q^4 + q^3 + q^2 + q + 1  & 18  \\
\hline
 &  q^{20} + q^{19}+ q^{18} + q^{17}+ q^{16} + q^{15}    &  \\ 
11 &   + q^5 + q^4 + q^3 + q^2 + q + 1  & 12  \\
\hline
 &  q^{22} + q^{21}+ q^{20} + q^{19}+ q^{18} + q^{17}  + q^{16} + q^{15} &   \\ 
 &  + q^{14} + 2 q^{13} + 2 q^{12} +  2q^{11}  + 2q^{10}  + 2 q^9 + q^8  &   \\ 
12 &   + q^7+ q^6 + q^5 + q^4 + q^3 + q^2 + q + 1  & 28  \\
\hline
\end{array}
\]
\end{table}

\begin{table}[ht]
\caption{\emph{The polynomials $B_n^{\circ}(q)$}}\label{tableB}
\renewcommand\arraystretch{1.3}
\noindent\[
\begin{array}{|c||c|c|c|}
\hline
n & B_n^{\circ}(q) & B_n^{\circ}(1) & B_n^{\circ}(-1) \\
\hline\hline
1 & 1 & 1 & 1 \\ 
\hline
2 & q+1 & 2 & 0 \\ 
\hline
3 & q^2+q & 2 & 0 \\ 
\hline
4 & q^3 + q^2+q & 3 & -1 \\ 
\hline
5 & q^4 + q^3 + q^2 - 1 & 2 & 0 \\ 
\hline
6 & q^5 + q^4 + q^3 + q^2 & 4 & 0 \\ 
\hline
7 &  q^6 + q^5 + q^4 + q^3 - q - 1 & 2 & 0 \\ 
\hline
8 &  q^7 + q^6 + q^5 + q^4 + q^3 - q & 4 & 0 \\ 
\hline
9 &  q^8 + q^7 + q^6 + q^5 + q^4 - q^2 - q & 3 & 1 \\ 
\hline
10 &  q^9 + q^8 + q^7 + q^6 + q^5 + q^4 - q^2 - q & 4 & 0 \\ 
\hline
11 &  q^{10} + q^9 + q^8 + q^7 + q^6 + q^5 - q^3 - 2q^2 - q & 2 & 0 \\ 
\hline
12 &  q^{11} + q^{10} + q^9 + q^8 + q^7 + q^6 + q^5 - q^3 - q^2 + 1 & 6 & 0 \\ 
\hline
\end{array}
\]
\end{table}

\begin{table}[ht]
\caption{\emph{The polynomials $A_n(q)$}}\label{tableA}
\renewcommand\arraystretch{1.3}
\noindent\[
\begin{array}{|c||c|c|c|}
\hline
n & A_n(q) & A_n(1) & A_n(-1) \\
\hline\hline
1 & q^2 & 1 & 1 \\ 
\hline
2 & q^4 + q^3 & 2 & 0 \\ 
\hline
3 & q^6 + q^5 + q^4 & 3 & 1 \\ 
\hline
4 & q^8 + q^7+ 2q^6 + q^5 & 5 & 1 \\ 
\hline
5 & q^{10} + q^9+ 2q^8 + 2q^7+ q^6 & 7 & 1 \\ 
\hline
6 & q^{12} + q^{11} + 2q^{10} + 3q^9+ 3q^8 + q^7 & 11 & 1 \\ 
\hline
7 &  q^{14} + q^{13} + 2q^{12} + 3q^{11} + 4q^{10} + 3q^9+ q^8 & 15 & 1 \\ 
\hline
8 &  q^{16} + q^{15} + 2q^{14} + 3q^{13} + 5q^{12} + 5q^{11} + 4q^{10} + q^9 & 22 & 2 \\ 
\hline
 &  q^{18} + q^{17} + 2q^{16} + 3q^{15} +  &  &  \\ 
9 &  + 5q^{14} + 6q^{13} + 7q^{12} + 4q^{11} + q^{10} & 30 & 2 \\ 
\hline
 &  q^{20} + q^{19} + 2q^{18} + 3q^{17} + 5q^{16} +  &  &  \\ 
10 &  + 7q^{15} + 9q^{14} + 8q^{13} + 5q^{12} + q^{11} & 42 & 2 \\ 
\hline
 &  q^{22} + q^{21} +  2q^{20} + 3q^{19} + 5q^{18} +  &  &  \\
11 &  + 7q^{17} + 10q^{16} + 11q^{15} + 10q^{14} + 5q^{13} + q^{12} & 56 & 2 \\ 
\hline
 &   q^{24} + q^{23} +  2q^{22} + 3q^{21} +  5q^{20} + 7q^{19} +   &  &  \\
12 &  + 11q^{18} + 13q^{17} + 15q^{16} + 12q^{15} + 6q^{14} + q^{13} & 77 & 3 \\ 
\hline
\end{array}
\]
\end{table}



\begin{thebibliography}{99}

\bibitem{Ap}
T. M. Apostol, 
\emph{Introduction to analytic number theory},
Undergraduate Texts in Mathematics, Springer-Verlag, New York-Heidelberg,~1976.

\bibitem{BR1}
R.~Bacher, C.~Reutenauer,
\emph{The number of right ideals of given codimension over a finite field}, 
Noncommutative birational geometry, representations and combinatorics, 1--18, 
Contemp. Math., 592, Amer. Math. Soc., Providence, RI, 2013.

\bibitem{BR2}
R.~Bacher, C.~Reutenauer,
\emph{Number of right ideals and a $q$-analogue of indecomposable permutations}, 
Canad. J.~Math. 68 (2016), no. 3, 481--503.

\bibitem{Bo}
N.~Bourbaki, 
\emph{Alg\`ebre commutative}, Herman, Paris~1961
(English translation: \emph{Commutative algebra}, Chapters 1--7, Springer-Verlag, Berlin,~1989).

\bibitem{CHM} 
M. A.~de Cataldo, T.~Hausel, L. Migliorini,
\emph{Exchange between perverse and weight filtration for the Hilbert schemes of points of two surfaces},
J.~Singul.~7 (2013), 23--38.

\bibitem{CV}
A.~Conca, G.~Valla, 
\emph{Canonical Hilbert-Burch matrices for ideals of~$k[x,y]$},
Michigan Math.~J.\ 57 (2008), 157--172.

\bibitem{Eis}
D.~Eisenbud, \emph{Commutative algebra. With a view toward algebraic geometry}, 
Grad. Texts in Math., 150, Springer-Verlag, New York,~1995.

\bibitem{ES1}
G.~Ellingsrud, S. A. Str\o mme, 
\emph{On the homology of the Hilbert scheme of points in the plane},
Invent. Math. 87 (1987), no.~2, 343--352. 

\bibitem{Fi}
N. J. Fine, \emph{Basic hypergeometric series and applications},
Mathematical Surveys and Monographs, 27, Amer. Math. Soc., Providence, RI, 1988.

\bibitem{Fo}
J.~Fogarty,
\emph{Algebraic families on an algebraic surface}, Amer. J.~Math 90 (1968), 511--521. 

\bibitem{GS}
L.~G\"ottsche, W.~Soergel, 
\emph{Perverse sheaves and the cohomology of Hilbert schemes of smooth algebraic surfaces},
Math. Ann. 296 (1993), no. 2, 235--245.

\bibitem{Gr}
A.~Grothendieck, 
\emph{Techniques de construction et th\'eor\`emes d'existence en g\'eom\'etrie alg\'e\-brique.~IV. 
Les sch\'emas de Hilbert}, S\'eminaire Bourbaki, Vol.~6, Exp. No.~221, 249--276, 
W.~A. Benjamin, New York-Amsterdam,~1966.

\bibitem{HR}
T.~Hausel, F.~Rodriguez-Villegas, 
\emph{Mixed Hodge polynomials of character varieties. With an appendix by Nicholas M.~Katz},
Invent. Math. 174 (2008), no. 3, 555--624. 

\bibitem{HLR0}
T.~Hausel, E.~Letellier, F.~Rodriguez-Villegas,
\emph{Topology of character varieties and representation of quivers},
C.~R. Math. Acad. Sci. Paris 348 (2010), no. 3--4, 131--135.

\bibitem{HLR1}
T.~Hausel, E.~Letellier, F.~Rodriguez-Villegas,
\emph{Arithmetic harmonic analysis on character and quiver varieties}, 
Duke Math.~J. 160 (2011), no.~2, 323--400.

\bibitem{HLR2}
T.~Hausel, E.~Letellier, F.~Rodriguez-Villegas, 
\emph{Arithmetic harmonic analysis on character and quiver varieties~II}, Adv. Math. 234 (2013), 85--128.

\bibitem{HW}
G. H. Hardy, E. M. Wright, 
\emph{An introduction to the theory of numbers}, 3rd~ed., Clarendon Press, Oxford,~1954.

\bibitem{KRinf}
C.~Kassel, C.~Reutenauer,
\emph{Complete determination of the zeta function of the Hilbert scheme of $n$~points on a two-dimensional torus},
arXiv:1610.07793.

\bibitem{OEIS}
\emph{The On-Line Encyclopedia of Integer Sequences}, published electronically at http://oeis.org.

\bibitem{Re}
M.~Reineke,
\emph{Cohomology of noncommutative Hilbert schemes}, Algebr.\ Represent.\ Theory~8 (2005), no.~4, 541--561.


\end{thebibliography}
\end{document}